\documentclass[12pt,a4paper]{amsart}
\usepackage{amsmath}
\usepackage{amsthm}
\usepackage{amssymb}
\usepackage{mathrsfs}
\usepackage[english]{babel}
\usepackage{ot1patch}

\newtheorem{thm}{Theorem}[section]

\newtheorem{cor}[thm]{Corollary}
\newtheorem{prop}[thm]{Proposition}

\theoremstyle{remark}
\newtheorem{rmk}[thm]{Remark}
\newtheorem*{rmk*}{Remark}

\theoremstyle{definition}
\newtheorem{dfn}[thm]{Definition}
\newtheorem{ex}[thm]{Example}

\numberwithin{equation}{section}

\newcommand{\ch}{{\mathcal{O}_c}}
\newcommand{\C}{\mathbb{C}}

\title{Algebraic normalisation}

\author{Adam Bia\l o\.zyt}\address{Jagiellonian University, Faculty of Mathematics and Computer Science, Institute of Mathematics, \L ojasiewicza 6, 30-348 Krak\'ow, Poland}\email{adam.bialozyt@doctoral.uj.edu.pl}
\date{November 30th 2023}
\keywords{Complex analytic and algebraic sets, c-holo\-morph\-ic
functions, Normalisation}
\subjclass{32B15, 32A17, 32A22}

\begin{document}
\begin{abstract}
We present a strictly geometric c-algebraic version of the analytic set normalisation. With the introduced tool we prove the Nullstellensatz for c-algebraic functions and study the growth exponent of a c-algebraic function.    
\end{abstract}

\maketitle

\section{Introduction}

One of the most elegant theorems in analytic geometry is the following statement.

\begin{thm}
Every analytic set $A$ admits a normalisation.
\end{thm}

The theorem is brief and concise, yet one must unpack a page-long list of definitions to understand its meaning fully.

Firstly, just for the sake of completeness, let us recall that we consider $A\subset \mathbb{C}^m$ to be analytic if it is given locally as an intersection of zeros of holomorphic functions. 

Secondly, we call an analytic set $Y\subset\mathbb{C}^m$ normal if every holomorphic function defined in its regular part $f: Reg\; Y\to \mathbb{C}$ which is locally bounded near singular points of $Y$ (we call such functions weakly holomorphic), extends holomorphically to a function $F$ defined on the whole $\mathbb{C}^m$. 

Finally a mapping $\pi: Y\to A$ is called a normalisation of an analytic set $A$ if
\begin{enumerate}
    \item The set $Y$ is a normal analytic set.
    \item The mapping $\pi$ is a finite holomorphism.
    \item The inverse image $\pi^{-1}(Reg\; A)$ is dense in $Y$.
    \item The restriction $\pi\mid \pi^{-1}(Reg\; A)$ is a biholomorphism.
\end{enumerate}

A proof of Theorem 1.1 can be found for example in Łojasiewicz's book \cite{Loj}. Normal sets are exactly the ones where the ring of weakly holomorphic functions is the same as the ring of holomorphic ones. In general, however, the holomorphic ring is strictly included in the weak holomorphic one and there is even one more well-renowned class placed between them. The intermediate step between holomorphic and weakly holomorphic functions 
comes from R. Remmert and consists of functions holomorphic on the regular part of $A$ and continuous on $A$.

\begin{dfn} (cf. \cite{R}) A mapping $f\colon
A\to{\C}^n$ is called {\it c-ho\-lo\-morph\-ic} if it is
continuous and the restriction of $f$ to the subset of regular
points $Reg\; A$ is holomorphic. We denote by $\ch (A)$ the ring of
c-holomorphic functions. 
\end{dfn}

Judging from the geometric perspective, c-holomorphic functions form a convenient generalisation of holomorphic mappings. For one, their graphs are always analytic sets. A natural subclass of c-holomorphic functions that we can expect to behave like polynomials or regular functions is formed by those c-holomorphic functions whose graphs are algebraic. 

\begin{dfn}
We will call {\it c-algebraic} any c-holomorphic function $f\colon A\to{\C}$ having an algebraic graph, where $A\subset{\C}^m$ is a fixed analytic set. 
We will denote by $$\mathcal{O}_c^\mathrm{a}(A)=\{f\in\mathcal{O}_c(A)\mid \Gamma_f\> \textrm{is algebraic}\}$$ the ring of such functions. We will also call $f=(f_1,\dots, f_n)$ c-algebraic if all functions $f_i$ are c-algebraic.

\end{dfn}
\begin{rmk} Of course, by the Chevalley-Remmert Theorem 
$$
{\mathcal{O}_c^a}(A)\neq\varnothing\>\Rightarrow\> A\>\textrm{is algebraic}.
$$
\end{rmk}

Readers interested in real algebraic geometry can see c-algebraic functions as a kind of complex counterpart of the recently introduced {\it regulous} functions \cite{FHMM}, \cite{Kol} in connection with \cite{Krz}.

Let us proceed with refinements of the classical normalisation theorem particularly interesting for us when working with c-algebraic functions.

Firstly, to adhere to the algebraic spirit of c-algebraic functions we will demand algebraic normal sets to constitute the following property.

\begin{dfn}
    We call an algebraic set $Y\subset \mathbb{C}^m$ a-normal if every c-algebraic function $f$ defined on $Y$ extends to a polynomial defined on $\mathbb{C}^m$.
\end{dfn}

\begin{ex}
    Let us remark swiftly that not every algebraic set is a-normal. Indeed, the primordial example of a set $A=\{y^2=x^3\}$ together with a function $f(x,y)=y/x, x\neq 0, f(0,0) = 0$ proves the point. The function $f$ is c-algebraic yet its derivative is unbounded near zero which cannot happen for the restrictions of polynomials.  
\end{ex}

The algebraisied definition of a normalisation should take the form of

\begin{dfn}\label{aNormDef}
We call $\pi: Y\to A$ an a-normalisation of an algebraic set $A$ if it satisfies the following conditions
\begin{enumerate}
    \item The set $Y$ is an a-normal algebraic set.
    \item The function $\pi$ is a bijective c-algebraic function. 
    \item The preimage $\pi^{-1}(Reg A)$ is dense in $Y$. 
    \item The restriction $\pi\mid\pi^{-1}(Reg A)$ is a biholomorphism. 
\end{enumerate}
\end{dfn}

With these definitions, we would like to recreate the Theorem 1.1:

\begin{thm}\label{a-normalisation}
    Every algebraic set $A\subset \mathbb{C}^m$ admits an a-normalisation.
\end{thm}

\begin{rmk}
    The condition $(3)$ from Definition~\ref{aNormDef} is superfluous. Since $\pi$ is postulated to be c-algebraic, the preimage of $RegA$ is always dense in $Y$. Indeed, the continuity of $\pi$ asserts that for every neighbourhood $U$ of a point $\pi(y)\in A$ we can find a neighbourhood $V\ni y$ such that $\pi(V)\subset U$. Now, since $RegA$ is dense in $A$, we can find a point $a\in RegA\cap \pi(V)$ thus $\pi^{-1}(a)$ lies in $V$. We still include the condition $(3)$ in the Definition~\ref{aNormDef} to maintain a bijection between conditions of normalisation and a-normalisation.
\end{rmk}

Now, are a-normal sets normal? Not necessarily: Observe that an algebraic set $A=\{xy=0\}$ is a-normal. Indeed, if $f\in\mathcal{O}_c^a(A)$ then, using Serre's algebraic graph theorem, we can easily find polynomials $P(x)$ and $Q(y)$ such that $f(x,y) = xP(x) + yQ(y) + f(0)$. However, $A$ cannot be normal as it is reducible.

Are on the other hand normal algebraic sets a-normal? They are! For, any c-algebraic function is a restriction of a rational function (see Theorem~\ref{rational repres}) and any rational function on a normal algebraic set is a restriction of a polynomial.

\section{Universal denominator}
The main tool used in the construction of normalisations is the so-called universal denominator. 

\begin{dfn}
    We call a function $Q$ a universal denominator for $A$ iff $Q$ does not vanish on any component of $A$ and $Qf$ is holomorphic for every $f\in\mathcal{O}_w(A)$
\end{dfn}

We can construct the universal denominator locally for any pure dimensional analytic set. It can be retrieved from a special description of the set.

\begin{prop}\label{universal den anal}
Let $A\subset U\times\mathbb{C}_t\times\mathbb{C}^{m-k}_y$ be a purely $k-$dimensional analytic set, where $U\subset \mathbb{C}^k_x$ is open and connected. Assume the natural projection $\pi:U\times\mathbb{C}\times\mathbb{C}^{m-k}\ni (x,t,y)\to x\in U$ to be proper on $A$ with a covering number $d$. Then, possibly after a rotation in $(t,y)$ coordinates, there exists a unitary polynomial $P\in\mathcal{O}(U)[t]$ of degree $d$ such that $Q(x,t,y)=\frac{\partial P}{\partial t}(x,t)$ is a universal denominator for $A$.
\end{prop}
\begin{proof}
    
Let $\rho(x,t,y)=(x,t)$ and $\xi(x,t)=x$. Surely, both $\rho$ and $\xi$ must be proper on $A$, and $\rho(A)$ function $\pi$ being such. The projection $\rho(A)$ is therefore an analytic set of codimension 1. Moreover, for a generic $x\in U$ we have $\#\pi^{-1}(x)=d$, thus picking $x\in U$ we can rotate the coordinates $(t,y)$ in such a way that \[\pi^{-1}(x)=\{(x,t_1(x),y_1(x)),\ldots,(x,t_d(x),y_d(x))\}\] with $t_i(x)\neq t_j(x),i\neq j$  and the multiplicity of $\xi$ on $\rho(A)$ is equal $d$. Then $(\rho(A),\xi)$ becomes an analytic cover of $\mathbb{C}^k$ of multiplicity $d$, so we can describe $\rho(A)$ by a unitary polynomial $P(x,t)=\prod_{i=1}^d(t-t_i(x)) \in \mathcal{O}(U)[t]$.

We will prove now that the polynomial $P$ satisfies the theorem assertion.
For any point $x$ in a simply connected set $V$ not intersecting the critical set of $P$ we have  $(t_1(x), y_1(x)), \ldots , (t_d(x), y_d(x))$, exactly $d$
distinct points over $x$ in $A$. Given any $f\in\mathcal{O}_w(A)$ set 
\[h(x,t,y)=\sum_{i=1}^df(x,t_i(x),y_i(x))\prod_{i\neq j}(t-t_j(x)).\]
Clearly $h(x,t_i(x),y_i(x))=f(x,t_i(x),y_i(x))Q(x,t_i(x))$ on $Reg A$, where $Q(x,t,y)=\frac{\partial P}{\partial t}(x,t)$. Now, $h$ is locally holomorphic apart from the critical set of $f$, and locally bounded near the critical points of $P$ and $f$. Therefore we can use the Riemann theorem to extend $h=Qf$ to a holomorphic function defined on $U\times \mathbb{C}\times\mathbb{C}^{m-k}$.
\end{proof}

Since any algebraic set can be rotated in a way the natural projection becomes proper on it, the universal denominator from Proposition~\ref{universal den anal} is a global one. In case we are dealing with c-algebraic function we can obtain an even stronger result.

\begin{thm}\label{rational repres}
    Let $A\subset \mathbb{C}^m$ be a purely dimensional algebraic set. Then there exists a polynomial $Q\in\mathbb{C}[x_1,\ldots,x_m]$ such that any c-algebraic function $f\in\mathcal{O}_c^a(A)$ can be represented as $f=\frac{R}{Q}$ where $R\in\mathbb{C}[x_1,\ldots,x_m]$.
\end{thm}
\begin{proof}
    If $\dim A = m$ then both Serre algebraic graph theorem and Liouville theorem give us $\mathcal{O}_c^a(A)=\mathbb{C}[x]$ so there is nothing to prove.

    Therefore, let us assume $\dim A = k < m$, $\pi(x,y)\to x\in \mathbb{C}^k$ be proper on $A$ and take $\rho, P, Q$, as in the proof of the previous proposition. Due to Chevalley theorem, we can observe that $\rho(A)$ is algebraic and $1-$codimensional, so the unitary polynomial $P$ is a pure-bred polynomial. Clearly, $Q$ is now a polynomial as well. Following Proposition~\ref{universal den anal} there exists a holomorphic function $h\in\mathcal{O}(\mathbb{C}^m)$ such that \[h=fQ \text{ on }Reg A.\] Since $f=h/Q$ is continuous, we can see that the quotient $h/Q$ is continuous on $A$. 
    
    Now we need to show, that we can take $h$ to be a polynomial. Indeed, $fQ$ is a holomorphic function in $\mathbb{C}^m$ with an algebraic graph over the algebraic set $A$. Thus it extends polynomially to the whole space $\mathbb{C}^m$ due to Serre’s algebraic graph theorem (see \cite{Loj}).
\end{proof}

The main consequence of the universal denominator existence in our case is that the morphism $\Phi: \mathcal{O}_c^a(A)\ni f\to Qf\in \mathbb{C}[z]$ is a monomorphism of rings. Taking into account that the ring of complex polynomials is Noetherian, we get that the ring of c-algebraic functions is finitely generated as a module over $\mathbb{C}[z]$. 

\section{A-normalisation}

We can finally prove Theorem~\ref{a-normalisation}. We will adapt the strategy of Łojasiewicz \cite{Loj} and construct the desired a-normalisation explicitly. 

\begin{thm}
    Let $A\subset \mathbb{C}^m$ be a purely dimensional algebraic set. Take $h_1,\ldots,h_r\in\mathcal{O}_c^a(A)$ the generators of $\mathcal{O}_c^a(A)$ as a module over $\mathbb{C}[z]$ and set \[H:=\{(z,w)\in\mathbb{C}^m\times\mathbb{C}^r\mid z\in A, \, w_i=h_i(z)\}.\]
    Then the natural projection $\pi: H\ni (z,w)\to z\in A$ is an a-normalisation for $A$.
\end{thm}
\begin{proof}
    Surely $\pi$ is a bijective c-algebraic function as it is a projection from the graph of the map $(h_1,\ldots,h_r)$. The preimage $\pi^{-1}(Reg A)$ is dense in $H$ because all $h_i$ are continuous and $RegA$ is dense in $A$. Finally, the restriction $\pi\mid\pi^{-1}(Reg A)$ is a biholomorphism because all $h_i$ restricted to $Reg A$ are holomorphic. We have to check that $H$ is a-normal.

    The algebraicity of $H$ follows again from the fact that H is a graph of a c-algebraic mapping $(h_1,\ldots,h_r)$. Let us take now $f\in\mathcal{O}_c^a(H)$; we need to check if it is a restriction of a polynomial defined in $\mathbb{C}^m\times\mathbb{C}^r$.

    To that end, observe that $A \ni z\to f(z,h(z))\in \mathbb{C}$ is a composition of two c-algebraic functions $f$ and $(id_x,h)$ and as such it is c-algebraic as well. Being so we can find polynomials $p_i\in \mathbb{C}[z]$ such that 
    \[f(z,h(z))=\sum_{i=1}^r p_i(z)h_i(z).\]
    Define now a polynomial $F(z,w):=\sum p_i(z)w_i$. We can see that for $(z,w)\in H$ we have:
    \[F(z,w)=F(z,h(z))=\sum_{i=1}^r p_i(z)h_i(z)=f(z,h(z))=f(z,w).\]
    That ends the proof
\end{proof}

\begin{dfn}
    We will call the set $H$ from the previous theorem (i.e. the graph of an arbitrary set of generators of $\mathcal{O}_c^a$) a canonical a-normalisation of $A$ and denote it by $\hat{A}$.
\end{dfn}

\begin{rmk}
    The proof of a-normalisation existence is simpler than the original Lojasiewicz proof of normalisations for analytic sets. All thanks to global definitions of the universal denominator and functions in $\mathcal{O}_c^a(A).$ We were able to omit difficulties arising in showing the proper class of the set $H$ and we didn't need to glue up local normalisations.
\end{rmk}


\begin{prop}
    Let $\pi_1:Y_1\to A$, $\pi_2:Y_2\to A$ be two a-nor\-ma\-li\-sa\-tions of $A$. Then, there exists a polynomial $P$, uniquely determined on $Y_1$, such that $P|Y_1=\pi_2^{-1}\circ\pi_1$. 
\end{prop}
\begin{proof}
    An idea of a proof is to observe that $\pi_2^{-1}\circ\pi_1$ is a composition of two c-algebraic functions. Being so it is also a c-algebraic function defined on $Y_1$. Therefore, it extends to a polynomial $P$ with $P|Y_1=\pi_2^{-1}\circ\pi_1$. Yet there is a problem here. The definition of the a-normalisation imposes on $\pi_i$ only holomorphicity. A stronger c-algebraic condition is still to be proved.   
\end{proof}

\section{Nullstellensatz}

A-normalisations give us a short and natural proof of the Nullstellensatz. 

\begin{thm}
    Let $A\subset \mathbb{C}_x^k$ be a purely dimensional algebraic set and $g,f_1,\ldots,f_r\in\mathcal{O}_c^a(A)$ be such that $g^{-1}(0)\supset \bigcap f_i^{-1}(0)$. Then there exist functions $q_i\in\mathcal{O}_c^a(A)$ and a natural number $n>0$ such that 
    \[g(x)^n=\sum_{i=1}^r q_i(x)f_i(x).\]
\end{thm}
\begin{proof}
    Consider the canonical a-normalisation $\hat{A}\subset \mathbb{C}_x^k\times\mathbb{C}_t^m$. Then the pullbacks $\hat{g}=g\circ \pi$ and $\hat{f_i}=f_i\circ\pi$ are all c-algebraic on $\hat{A}$. Therefore, we can find polynomials $G, F_i$ extending $\hat{g},\hat{f_i}$ respectively. Since $\hat{A}$ is algebraic we can also find polynomials $P_i$ such that $A=\bigcap P^{-1}(0)$. 
    
    Now, observe that $G^{-1}(0)\supset \bigcap_{i=1}^r F^{-1}_i(0)\cap \bigcap_{j=1}^d P^{-1}_j(0)$. By the traditional Nullstellensatz we can find complex polynomials $w_i, v_i$ and a natural number $n>0$ such that 
    \[G^n=\sum_{i=1}^r w_iF_i+\sum_{j=1}^d v_jP_j.\]
    On $\hat{A}$ we have then $\hat{g}^n=G^n=\sum w_iF_i=\sum w_i\hat{f_i}$. Placing $h_j(x)$ (the $j-th$ generator of $\mathcal{O}_c^a(A)$) in place of $t_j$ coordinate produces c-algebraic functions $q_i=w_i\circ \pi^{-1}\in\mathcal{O}_c^a(A)$ such that \[g(x)^n=\sum_{i=1}^r q_i(x)f_i(x).\]
\end{proof}

Clearly, we have the 'weak' consequence.

\begin{cor}
    Let $A$ be a purely dimensional algebraic set and $f_1,\ldots,f_r\in\mathcal{O}_c^a(A)$ have no common zeros. Then, they generate the unit ideal.
\end{cor}

Let us note briefly that the effective Nullstellensatz for the ring of c-algebraic functions was obtained recently in \cite{BDT} using a completely different method. The approach based on a-normalisations presented in the article, however, despite its weaker result, has the advantage of a more straightforward proof. 

\section{The canonical a-normalisation and the growth exponent}

Let $f\in\mathcal{O}_c^a(A)$. We can pull $f$ back on the canonical a-normalisation $\hat{A}$ by putting $\hat{f}(x,y)=f\circ \pi (x,y),$ where $\pi:\hat{A}\to A$ is a natural projection.

Let us recall that in \cite{BDTW} we defined a growth exponent for a c-algebraic function:

\begin{dfn}\label{growth exponent}
    For a function $f\in\mathcal{O}_c^a(A)$, we call a number
\[ \mathcal{B}(f):=\inf\{s\geq 0\mid |f(x)|\leq \mathrm{const.}(1+\| x\|)^s,\ x\in A\} \]
    the growth exponent of $f$.
\end{dfn}
The growth exponent $\mathcal{B}(f)$ is always a rational number and is always attained. Moreover, it does not depend on the choice of the norm in the definition; being so we will work with the maximum norm.
The growth exponent plays the role of a degree of a polynomial in the c-algebraic functions ring (see \cite{BDTW, BDT, BDT2}).

It is interesting, how the growth exponents of $f$ and $\hat{f}$ behave.

Firstly, set $\mathcal{B}(f)=b$ and observe that we can simply bound as follows:
\[\|\hat{f}(x,y)\|=\|f(x)\|\leq C(1+\|x\|)^b\leq C(1+\|x\|+\|y\|)^b=C(1+\|(x,y)\|)^b.\]

Therefore, we obtain an inequality:  $\mathcal{B}(\hat{f})\leq \mathcal{B}(f)$.
On the other hand if we take any $(x,y)\in\hat{A}$ and $C,s>0$ we can write:
\[C(1+\|x\|+\|y\|)^s=C(1+\|x\|+\sum \|h_i(x)\|)^s\]

Where $h_i$ are the generators of $\mathcal{O}_c^a(A)$, thus we can estimate further
 
 \[C(1+\|x\|+\sum \|h_i(x)\|)^s \leq C(1+\|x\|+\sum c_i(1+\|x\|)^{\mathcal{B}(h_i)})^s.\]

 Now changing the constant we can rewrite the last formula as
 \[C(1+\|x\|)^{s\cdot \max\{1,\mathcal{B}(h_i)\}}\]
 and observe that should $s\cdot\max\{1,\mathcal{B}(h_i)\}<\mathcal{B}(f)$ we would find $x\in A$ such that \[C(1+\|x\|)^{s\cdot \max\{1,\mathcal{B}(h_i)\}}<\|f(x)\|=\|\hat{f}(x,y)\|.\]
 We have proven therefore

 \begin{prop}
     Let $\hat{A}$ be a canonical a-normalisation of $A$. Then for any $f\in\mathcal{O}_c^a(A)$ and its pullback  $\hat{f}$ we have
     \[ \mathcal{B}(f)\geq \mathcal{B}(\hat{f})\geq \frac{\mathcal{B}(f)}{\max\{1,\mathcal{B}(h_i)\}}.\]
     Where $h_i$ are the generators of $\mathcal{O}_c^a(A)$.
 \end{prop}

Now the question remains whether $\mathcal{B}(h_i)$ ever exceed one.

\section{Algebraic approach}
A similar idea exists in Francois Bernard's article \cite{Bernard}, and in A.....
Their approach however is purely algebraic. Bernard's seminormalisations are constructed implicitly as the biggest set on the way to the normalisation of $X$ with bijective morphism.  With algebraic properties at hand, Bernard proves a series of results for semi-normalised sets. It will be beneficial to compare both definitions and see if they lead to the same object. Our approach would give then a tangible construction of the seminormalisation.

\section{Acknowledgements}

The author would like to sincerely thank Professor Maciej Denkowski for pointing out the example of a-normalistaion's application in Nullstellensatz proof. And all the participants of KonfRogi for the discussion that helped to develop this article.


\begin{thebibliography}{FHMM}

\bibitem[ATW]{ATW} R. Achilles, P. Tworzewski, T. Winiarski, {\it On improper isolated
intersection in complex analytic geometry},
Ann. Polon. Math. LI (1990), 21-36;


\bibitem[BDTW]{BDTW} A. Bia\l o\.zyt, M. P. Denkowski, P. Tworzewski, T. Wawak, {\it On the growth exponent of c-holomorphic functions with algebraic graphs}, {\tt arXiv:1406.5293},  submitted 2020;


\bibitem[BDT]{BDT} A.Bia\l o\.zyt, M. P. Denkowski, P. Tworzewski, {\it On the Nullstellensatz for c-holomorphic functions with algebraic graphs}, Ann. Polon. Math. 125 (2020), 1-11;


\bibitem[BDT2]{BDT2} A.Bia\l o\.zyt, M. P. Denkowski, P. Tworzewski, {\it The \L ojasiewicz exponent at infinity of c-holomorphic functions with algebraic graphs}, preprint 2020;

\bibitem[B]{Bernard} F. Bernard, {\it Seminormalization and regulous functions on complex affine varieties}, 
      {\tt arXiv:2109.06542}, submitted 2022;

\bibitem[CNP]{CNP} P. Cassou-Nogu\`es, A. P\l oski, {\it Un th\'eor\`eme des z\'eros effectif}, Bull. Pol. Acad. Sci. Math. 44 (1996), 61-70;

\bibitem[Cg]{Cg} E. Cygan, {\it Nullstellensatz and cycles of zeroes of holomorphic mappings}, Ann. Polon. Math. 78 no.2 (2002), 181-191;



\bibitem[Ch]{Ch} E. M. Chirka, {\it Complex Analytic Sets}, Kluwer Acad. Publ. 1989;

\bibitem[D1]{D} M. P. Denkowski, {\it The \L ojasiewicz exponent of c-holomorphic mappings}, Ann. Polon. Math. LXXXVII.1 (2005), pp. 63-81;

\bibitem[D2]{D2} M. P. Denkowski, {\it A note on the Nullstellensatz for c-holomorphic functions}, Ann. Polon. Math. XC.3 (2007), 219-228;


\bibitem[D3]{Dm} M. P. Denkowski, {\it Multiplicity and the pull-back problem}, Manuscripta Math. 149 (2016), 83-91;


\bibitem[Dr]{Dr} R. N. Draper, {\it Intersection theory in analytic geometry}, Math. Ann. 180
(1969), 175-204;

\bibitem[FPT]{FPT} A. Fabiano, A.P\l oski, P.Tworzewski, {\it Effective Nullstellensatz for strictly regular sequences}, Univ Iagell. Acta Math. 38 (2000), 163-167;

\bibitem[FHMM]{FHMM} G. Fichou, J. Huisman, F. Mangolte, J.-P. Monnier, {\it Fonctions r\'egulues}, Journ. Reine Angew. Math. 718 (2016), 103-151;








\bibitem[Kol]{Kol} J. Koll\'ar, {\it Continuous rational functions on real and p-adic varieties}, arXiv
1101.3737 [math.AG] (2011);


\bibitem[K]{Krz} W. Kucharz, {\it Rational maps in real algebraic geometry}, Adv. Geom. 9 (4)
(2009), 517–539;

\bibitem[\L]{Loj} S. {\L}ojasiewicz, {\it Introduction to Complex Analytic Geometry}, Birkh\"auser, Basel
1991;

\bibitem[P2]{P2} A. P{\l}oski, {\it Multiplicity and the {\L}ojasiewicz exponent}, Singularities,
Banach Center Publ. vol. 20 PWN Warsaw (1988), 353-364;

\bibitem[PT]{PT} A. P{\l}oski, P. Tworzewski, {\it Effective Nullstellensatz on analytic and
algebraic varieties}, Bull. Pol. Acad. Sci. Math. vol. 46 nr 1 (1998), 31-38;

\bibitem[R]{R} R. Remmert, {\it Projektionen analytischer Mengen}, Math. Ann. 130 (1956), 410-441;
%
\bibitem[Sp]{Sp} S. Spodzieja, {\it The {\L}ojasiewicz exponent at infinity for overdetermined polynomial mappings}, Ann. Pol. Math. LXXVIII.1
(2002), 1-10;
%
%



\bibitem[TW1]{TW1} P. Tworzewski, T. Winiarski, {\it Analytic sets with proper projections}, Journ. Reine Angew. Math. 337 (1982), 68-76;

\bibitem[TW2]{TW2} P. Tworzewski, T. Winiarski, {\it Cycles of zeroes of holomorphic mappings}, Bull. Pol. Acad. Sci. Math. vol. 37 (1986), 95-101;


\bibitem[Wh]{Wh} H. Whitney, {\it Complex Analytic Varieties}, Addison-Wesley Publ. Co. 1972.

\end{thebibliography}
\end{document}